\documentclass[11pt,pdftex]{amsart}
\usepackage{txfonts,mathptmx}
\usepackage{mathrsfs}
\usepackage[colorlinks=true]{hyperref}
\usepackage{amsrefs}

\usepackage{tikz-cd}

\numberwithin{equation}{section}

\newtheorem{thm}{Theorem}

\newtheorem*{cor-nn}{Corollary}

\theoremstyle{definition}

\theoremstyle{remark}

\newtheorem*{ack}{Acknowledgement}

\newcommand{\Q}{\mathbb{Q}}
\newcommand{\R}{\mathbb{R}}
\newcommand{\Z}{\mathbb{Z}}

\begin{document}
\title[Wirtinger presentations and hyperbolic groups]{
Groups admitting Wirtinger presentations \\ and Gromov hyperbolic groups}
\author[T. Akita]{Toshiyuki Akita}
\address{Department of Mathematics, Faculty of Science, Hokkaido University,
Sapporo, 060-0810 Japan}
\email{akita@math.sci.hokudai.ac.jp}

\begin{abstract}
Twisted Wirtinger presentations are generalizations of the classical Wirtinger presentations of knot and link groups.
In this paper, we prove that if a finitely generated group admitting a twisted Wirtinger presentation is Gromov hyperbolic, then its second rational homology group vanishes. Moreover, if the group is torsion-free, then its second integral homology group also vanishes.
\end{abstract}

\maketitle
\section{Introduction}
We say that a group presentation $\langle X\mid R\rangle$ is a 
\emph{twisted Wirtinger presentation} if
each relation in $R$ is of the form 
\[
w^{-1}xw=y^\epsilon\quad (x,y\in X, w\in F(X),\epsilon\in\{\pm 1\}),
\]
where $F(X)$ is the free group on $X$.
In particular, if all the relations in $R$ are of the form
\[
w^{-1}xw=y\quad (x,y\in X, w\in F(X)),
\]
then we say that $\langle X\mid R\rangle$ is a \emph{Wirtinger presentation}.
Both Wirtinger and twisted Wirtinger presentations generalize 
the classical Wirtinger presentations of knot and link groups.
Groups admitting Wirtinger presentations are known by various names.
Among others, they are called {Wirtinger groups} in \cite{Yajima},
{labelled oriented graph groups}
({LOG groups} in short) in \cite{Howie},
and {C-groups} 
in \cites{Kulikov,MR1307063}.
Groups admitting twisted Wirtinger presentations are referred to as non-orientable C-groups in \cite{MR1392843}, where a characterization of these groups in terms of group homology was established.

Examples of groups that admit Wirtinger presentations include free and free abelian groups,
braid groups and pure braid groups, 
Artin groups, Thompson's group $F$, and higher-dimensional knot groups.
As for twisted Wirtinger presentations, examples include 
the infinite dihedral group, elementary abelian $2$-groups,
Coxeter groups, twisted Artin groups, and cactus groups.

More generally, let $M$ be a closed smooth $n$-dimensional manifold smoothly embedded
in $\R^{n+2}$. Then the fundamental group of the complement $\R^{n+2}\setminus M$
admits a twisted Wirtinger presentation, and
it admits a Wirtinger presentation if $M$ is orientable \cite{MR1841759}.

Gromov hyperbolic groups---also known as word hyperbolic or simply hyperbolic groups---were introduced and developed by Gromov~\cite{MR0919829}, and are among the most fundamental objects in geometric group theory.
Roughly speaking, a finitely generated group is Gromov hyperbolic 
if its Cayley graph behaves like a negatively curved space.
For the precise definition and properties of Gromov hyperbolic groups,
see \cites{MR1744486,MR1086648} for instance.

Free groups $F_n$ are Gromov hyperbolic, whereas free abelian groups $\Z^n$ $(n\geq 2)$
are not.
Necessary and sufficient conditions for Coxeter groups to be Gromov hyperbolic were obtained
in~\cite{MR2636665} (see also~\cite{MR2360474}).
A large number of Coxeter groups are Gromov hyperbolic, while others are not.
These facts show that among finitely
generated groups admitting twisted Wirtinger presentations, some are Gromov hyperbolic whereas
others are not.

The purpose of this paper is to provide a simple group homological obstruction to Gromov hyperbolicity 
for finitely generated groups that admit twisted Wirtinger presentations,
which is a consequence of the following result:
\begin{thm}\label{thm:algebra}
Let $G$ be a group admitting a twisted Wirtinger presentation.
If $G$ does not contain a subgroup isomorphic to $\Z\times\Z$,
then the second rational homology group $H_2(G;\Q)$  is trivial.
Moreover, if the group is torsion-free,
then the second integral homology group $H_2(G;\Z)$  is also trivial.
\end{thm}
It is well-known that a Gromov hyperbolic group does not contain a subgroup 
isomorphic to $\Z\times\Z$
(cf.~\cites{MR1744486,MR1086648}).
The following corollary follows immediately from Theorem \ref{thm:algebra}:
\begin{cor-nn}\label{thm:main}
Let $G$ be a finitely generated group admitting a twisted Wirtinger presentation.
If $G$ is Gromov hyperbolic, then the second rational homology group $H_2(G;\Q)$  is trivial.
Moreover, if $G$ is both Gromov hyperbolic and torsion-free, 
then the second integral homology group $H_2(G;\Z)$  is also trivial.
\end{cor-nn}
We refer to~\cite{MR672956} for the standard facts about group homology.
It is worth noting that 
any finitely generated abelian group $A$
can be realized as the second homology group $H_2(G;\Z)$ of a 
group $G$ admitting a Wirtinger presentation
\cite{MR0635591}.
Hence, a large number of groups admitting twisted Wirtinger presentations
cannot be Gromov hyperbolic.
Conversely, many Gromov hyperbolic groups have positive second Betti number, 
and such groups cannot admit twisted Wirtinger presentations.

\section{Proof}
From now on, we write $H_2(G)$ in place of $H_2(G;\Z)$ for simplicity.
The proof of Theorem \ref{thm:algebra} relies on a result of Kuz\cprime min \cite{MR1392843}
concerning groups 
that admit twisted Wirtinger presentations.
Before quoting his result, we introduce some necessary preliminaries.
Let $G$ be a group and let $g,h\in G$ be commuting elements (i.e. $[g,h]=1$).
Define the integral homology class $g\wedge h\in H_2(G)$ by
\[
g\wedge h\coloneqq [g|h]-[h|g]\in H_2(G),
\]
where the right-hand side represents a normalized $2$-cycle of $G$.
This homology class is sometimes referred to as the \emph{Pontryagin product} of $g,h\in G$.
It is natural with respect to group homomorphisms. 
Namely, for a group homomorphism $f\colon G\to H$ and commuting elements $g,h\in G$, one has
\[
f(g)\wedge f(h)=f_*(g\wedge h)\in H_2(H).
\]
As a consequence of the naturality, 
for any commuting elements 
$g,h\in G$, consider the homomorphism
 $\phi\colon\Z\times\Z\to G$ defined by $\phi(1,0)=g$, $\phi(0,1)=h$. 
 Then the homology class $g\wedge h$ is exactly the image of 
$(1,0)\wedge (0,1)\in H_2(\Z\times\Z)$ under the induced homomorphism 
$\phi_*\colon H_2(\Z\times\Z)\to H_2(G)$, that is
\begin{equation}\label{eq:wedge}\tag{$\star$}
g\wedge h=\phi_*((1,0)\wedge (0,1))\in H_2(G).
\end{equation}
Note that the integral homology class $(1,0)\wedge (0,1)\in H_2(\Z\times\Z)$ 
can be shown to be a generator of
$H_2(\Z\times\Z)\cong\Z$.
For the properties of the Pontryagin product, see \cite{akita-even-Artin}.

We now present a result of Kuz\cprime min \cite{MR1392843}
concerning groups that admit twisted Wirtinger presentations, 
which will be used in the proof of Theorem \ref{thm:algebra}.
He characterized such groups in terms of their second integral homology.
The following theorem is a consequence of his result:

\begin{thm}[Kuz\cprime min \cite{MR1392843}]\label{thm:Kuzmin}
Let $G$ be a group admitting a twisted Wirtinger presentation.
For any homology class $u\in H_2(G)$, there exist commuting elements 
$g_i,h_i\in G$ $(1\leq i\leq m)$ such that
\[
u=\sum_{i=1}^m g_i\wedge h_i\in H_2(G).
\]
\end{thm}


\begin{proof}[Proof of Theorem \ref{thm:algebra}]
Let $G$ be a group. Suppose that
\begin{enumerate}
\item $G$ does not contain a subgroup isomorphic to $\Z\times\Z$,
\item $G$ admits a twisted Wirtinger presentation.
\end{enumerate}
Let $\phi\colon\Z\times\Z\to G$ be a group homomorphism, and let
$A\subset G$ be the image of $\phi$.
It follows from the condition (1) that $A$ is an abelian group generated by
two elements and of rank at most one.
For such an abelian group $A$, the second rational homology group $H_2(A;\Q)=H_2(A)\otimes\Q$
is trivial.
This implies that the induced homomorphism $\phi_*\colon H_2(\Z\times\Z;\Q)\to H_2(G;\Q)$ is the zero map,
since $\phi_*$ factors through $H_2(A;\Q)$ as 
\[
H_2(\Z\times\Z;\Q)\to H_2(A;\Q)\to H_2(G;\Q).
\]
In view of the identity \eqref{eq:wedge}, the rational homology class
$g\wedge h\in H_2(G;\Q)$ vanishes for any commuting elements $g,h\in G$.
On the other hand, by Theorem \ref{thm:Kuzmin}, 
every element of $H_2(G;\Q)=H_2(G)\otimes\Q$
is a sum of rational multiples of
elements of the form $g\wedge h\in H_2(G;\Q)$,
where $g,h\in G$ commute with each other.
Thus we conclude that $H_{2}(G;\Q)=0$.
This completes the proof of the first half of the theorem.

We now assume, in addition, that $G$ is torsion-free.
Let $A$ be the image of a homomorphism $\phi\colon\Z\times\Z\to G$ as before.
Then, by the condition (1), the subgroup $A$ must be either trivial 
or infinite cyclic.
For such a group $A$, the second integral homology $H_2(A)$ vanishes.
This implies that the induced homomorphism 
$\phi_{*}\colon H_{2}(\Z\times\Z)\to H_{2}(G)$ is the zero map, since
$\phi_*$ factors through $H_2(A)$ as 
\[
H_2(\Z\times\Z)\to H_2(A)\to H_2(G).
\]
In view of the identity \eqref{eq:wedge}, the integral homology class
$g\wedge h\in H_2(G)$ vanishes for any commuting elements $g,h\in G$.
The second half of the theorem now follows from Theorem \ref{thm:Kuzmin}.
\end{proof}

\begin{ack}
The author was partially supported by JSPS KAKENHI Grant Number 24K06727.
\end{ack}

\end{document}